\documentclass[11pt,reqno]{article}
\usepackage{latexsym, amsmath, amssymb, amsthm, a4, epsfig}
\usepackage{graphicx}
\usepackage{caption}
\usepackage{subcaption}
\usepackage{mathdots}
\usepackage{xcolor}
\usepackage{booktabs}
\usepackage{enumerate}

\newtheorem{theorem}{Theorem}[section]

\newtheorem{lemma}[theorem]{Lemma}
\newtheorem{prop}[theorem]{Proposition}

\newtheorem{remark}{Remark}

\newcommand{\ds}{\displaystyle}
\newcommand{\p}{\partial}

\newcommand{\eqnref}[1]{(\ref {#1})}

\newcommand{\Cbb}{\mathbb{C}}

\newcommand{\Rbb}{\mathbb{R}}

\newcommand{\Zbb}{\mathbb{Z}}

\newcommand{\Ccal}{\mathcal{C}}

\newcommand{\Hcal}{\mathcal{H}}

\newcommand{\Kcal}{\mathcal{K}}
\newcommand{\Dcal}{\mathcal{D}}

\newcommand{\Mcal}{\mathcal{M}}

\newcommand{\Ga}{\alpha}

\newcommand{\Gvf}{\varphi}

\newcommand{\Gl}{\lambda}

\newcommand{\Gm}{\mu}
\newcommand{\Gv}{\nu}
\newcommand{\Gp}{\pi}

\newcommand{\Gs}{\sigma}

\newcommand{\Gz}{\zeta}

\newcommand{\GG}{\Gamma}

\newcommand{\GO}{\Omega}

\newcommand{\GY}{\Psi}

\newcommand{\beq}{\begin{equation}}
\newcommand{\eeq}{\end{equation}}

\def\ol{\overline}

\numberwithin{equation}{section}
\numberwithin{figure}{section}

\begin{document}

\title{Spectral properties of the Neumann-Poincar\'e operator on rotationally symmetric domains in two dimensions\thanks{This work was partially supported by NRF (of S. Korea) grant No. 2019R1A2B5B01069967.}}

\author{Yong-Gwan Ji\thanks{Department of Mathematics and Institute of Applied Mathematics, Inha University, Incheon 22212, S. Korea (22151063@inha.edu, hbkang@inha.ac.kr).} \and Hyeonbae Kang\footnotemark[2] }

\maketitle

\begin{abstract}
This paper concerns the spectral properties of the Neumann-Poincar\'e operator on $m$-fold rotationally symmetric planar domains. An $m$-fold rotationally symmetric simply connected domain $D$ is realized as the $m$th-root transform of a certain domain, say $\GO$. We prove that the domain of definition of the Neumann-Poincar\'e operator on $D$ is decomposed into invariant subspaces and the spectrum on one of them is the exact copy of the spectrum on $\GO$. It implies in particular that the spectrum on the transformed domain $D$ contains the spectrum on the original domain $\GO$ counting multiplicities.
\end{abstract}

\noindent{\footnotesize {\bf AMS subject classifications}. 31A10 (primary), 35P05, 30B10 (secondary)}

\noindent{\footnotesize {\bf Keywords}. Neumann-Poincar\'e operator, spectrum, $m$th-root transform, $m$-fold rotational symmetry, Grunsky coefficients, lemmniscate, Cassini oval}

\section{Introduction}

The purpose of the present paper is to investigate the spectral properties of the NP operator (NP is the acronym of Neumann-Poincar\'e) on $m$-fold rotationally symmetric planar domains, which are domains invariant under the rotation by the angle $2\Gp/m$. An $m$-fold rotationally symmetric domain, if it is simply connected, is realized as the $m$th-root transform of a certain domain (see below). We are particularly interested in how the spectrum of the original domain is inherited by the $m$th-root transform.

The $m$th-root transform is defined as follows. Let $\GO$ be a simply connected bounded domain containing $0$ in its interior. By the Riemann mapping theorem there exist a positive constant $R$ and a univalent function $\GY$ from $\{|z|>R\}$ onto $\Cbb \setminus \overline{\GO}$ which admits a Laurent series expansion of the form
\beq\label{exter_Riemann}
\GY(z) = z + a_0 + \frac{a_1}{z} + \frac{a_2}{z^2} + \cdots.
\eeq
For a positive integer $m$, the $m$th-root transform of $\GY$ is defined to be
\beq\label{mthroot}
\GY_m(z) := \GY(z^m)^{1/m}
\eeq
for $|z|> R^{1/m}$. It is known (see \cite[p. 28]{duren}) that $\GY_m$ is univalent.

Let $\GO_m$ be the bounded domain defined by $\Cbb \setminus \overline{\GO_m} := \GY_m(\{|z|> R^{1/m}\})$. The domain $\GO_m$ is called the $m$th-root transform of $\GO$. It is easy to see that $\GO_m$ is $m$-fold rotationally symmetric. It is also easy to prove (actually it is left as an exercise in \cite{duren}) that any simply connected $m$-fold rotationally symmetric domain is an $m$th-root transform of a certain domain.

Let $\GO$ be a bounded domain with the Lipschitz continuous boundary. The NP operator (NP is the acronym of Neumann-Poincar\'e), denoted by $\Kcal_{\p \GO}$, is the boundary integral operator on $\p \GO$ defined by
\beq\label{Kcal}
\Kcal_{\p \GO}[\Gvf](x) := \int_{\p\GO} \p_{\Gv_y} \GG(x-y) \, \Gvf(y) \, d\Gs(y), \quad x \in \p \GO,
\eeq
where $\p_{\Gv_y}$ denotes the outward normal derivative at the point $y \in \p \GO$ and $\GG(x-y)$ is the fundamental solution to the Laplacian. In two dimensions which this paper is concerned with, $\GG(x) = \frac{1}{2\Gp} \ln |x|$. The integral above is understood as the Cauchy principal value if $\p\GO$ is merely Lipschitz continuous.

The NP operator appears naturally when solving the classical boundary value problem on $\GO$ in terms of layer potential, which was initiated by Neumann and Poincar\'e as the name of the operator suggests. Recently there is rapidly growing interest in the spectral properties of the NP operator in relation to the plasmonic resonance on meta material and significant new results are being produced. We refer to recent surveys \cite{AKMP, Kang} for historical account and recent development on the operator. In relation to the subject of this paper, we mention that $\Kcal_{\p \GO}$ can be realized as a self-adjoint operator on $\Hcal^{1/2}(\p \GO)$ ($\Hcal^{1/2}(\p \GO)$ is the Sobolev space of order $1/2$ on the curve $\p\GO$) by introducing a new inner product \cite{KPS}. If $\p \GO$ is $C^{1,\Ga}$-smooth for some  $\Ga >0$, then $\Kcal_{\p \GO}$ is compact on $\Hcal^{1/2}(\p \GO)$ and hence its spectrum consists of eigenvalues and their limit point $0$; If it has a corner, then $\Kcal_{\p \GO}$ admits nontrivial essential spectrum (see, for example, \cite{PP2}).

The following is the main result of this paper. Here and throughout this paper, $\Gs(\Kcal_{\p \GO}, \Hcal)$ denotes the spectrum of $\Kcal_{\p \GO}$ on $\Hcal$ for a subspace $\Hcal$ of $\Hcal^{1/2}(\p\GO)$ invariant under $\Kcal_{\p \GO}$. When $\Hcal=\Hcal^{1/2}(\p\GO)$, we denote it by $\Gs(\Kcal_{\p \GO})$.

\begin{theorem}\label{main}
Let $\GO$ be a simply connected domain with $C^{1,\Ga}$ boundary for some $\Ga>0$. Suppose $0 \in \GO$ and let $\GO_m$ be the $m$th-root transform of $\GO$. Then, there are subspaces $\Hcal_0, \ldots, \Hcal_{m_*}$ ($m_*=m/2$ if $m$ is even and $m_*=(m-1)/2$ if $m$ is odd) of $\Hcal^{1/2}(\p \GO_m)$ invariant under $\Kcal_{\p \GO_m}$ such that
\beq\label{decom1}
\Hcal^{1/2}(\p \GO_m)= \bigoplus_{j=0}^{m_*} \Hcal_j
\eeq
and
\beq\label{decom2}
\Gs(\Kcal_{\p \GO})= \bigcup_{j=0}^{m_*} \Gs(\Kcal_{\p \GO}, \Hcal_j).
\eeq
Moreover, there is a unitary transform $U$ from $\Hcal_{0}$ onto $\Hcal^{1/2}(\p\GO)$ such that
\beq\label{similar}
\Kcal_{\p \GO_m}|_{\Hcal_{0}} = U^{-1} \Kcal_{\p \GO} U,
\eeq
where $\Kcal_{\p \GO_m}|_{\Hcal_{0}}$ denotes the $\Kcal_{\p \GO_m}$ restricted to $\Hcal_{0}$, and hence
\beq\label{inclusion}
\Gs(\Kcal_{\p \GO}, \Hcal_0) = \Gs(\Kcal_{\p \GO}),
\eeq
counting multiplicities.
\end{theorem}

Above theorem says in particular that
$$
\Gs(\Kcal_{\p \GO}) \subset \Gs(\Kcal_{\p \GO_m}).
$$
If $m \ge 2$, then the inclusion is proper counting multiplicities since $\Hcal_j$ ($j \ge 1$) is not empty. There are domains for which the inclusion is proper as sets, namely, there is an eigenvalue on $\p \GO_m$ which is not an eigenvalue on $\p\GO$. For example, let $\GO$ be a disk centered at a point other than the origin. Then spectrum $\Gs(\Kcal_{\p \GO})$ consists of $0$ and $1/2$, where $0$ is an eigenvalue of infinite multiplicities and $1/2$ is a simple eigenvalue. The $m$th-root transform $\GO_m$ of $\GO$ is the $m$-leaf symmetric lemniscate and $\Gs(\Kcal_{\p \GO_m})$ contains infinitely many nonzero eigenvalues in addition to the eigenvalue $0$ of infinite multiplicities since it is known that if the NP operator on a bounded planar domain is of finite rank, then the domain must be a disk \cite[Theorem 7.6]{Shapiro}.

There are not many classes of domains where NP spectra (spectra of the NP operators) are known. Ellipses are among them (see, for example, \cite{KPS}). Since the NP spectrum is invariant under M\"obius transform \cite{Schiffer} (see also \cite{JK}), the NP spectrum on the lima\c{c}on of Pascal can be computed. The lima\c{c}on of Pascal is the image of the unit disc under the map $w = z + A z^2$ for some constant $A$ and can be realized as the M\"obius transform of an ellipse (see \cite{AKM}). On the other hand, it is proved in \cite{KPS} that lemniscates has $0$ as an NP eigenvalue and its multiplicity is infinite. Lemniscates have NP eigenvalues other than $0$ even though we don't know what they are. Theorem \ref{main} yields a way to construct domains, via $m$th-root transforms, whose partial NP eigenvalues can be computed.

Theorem \ref{main} is proved using the representation of the NP operator in terms of the Grunsky coefficients which is obtained in \cite{Jung and Lim}. In the next section we briefly review the Grunsky coefficients and present a simple alternative proof of the representation based on Cauchy's theorem. Using the representation we also give an alternative proof of the fact that the NP operator on lemniscates have the infinite dimensional kernel which was proved in \cite[Theorem 9]{KPS} as mentioned before. The proof of Theorem \ref{main} is presented in section \ref{proof}. In the section to follow we discuss the matrix representation of the NP operator on $m$-fold rotationally symmetric planar domains. In the last section we discuss two examples: $m$-star shaped domains and the Cassini oval. It is interesting, but seems quite difficult (if possible), to compute the spectrum on other subspaces $\Hcal_1, \ldots, \Hcal_{m_*}$. We discuss possibility of computing them on the Cassini oval, which seems the simplest possible case: The Cassini oval is the 2nd-root transform of a disk.

Throughout this paper, we use notation $z$ to denote a point in $\Cbb$ or a point in $\Rbb^2$, namely, $z=(z_1,z_2)$.

\section{The NP operator and Grunsky coefficients, lemniscates}\label{section NP Grunsky}

\subsection{Double layer potential and Cauchy transform}\label{subsec:cauchy}

Here and afterwards, we assume that $\GO$ be a simply connected bounded planar domain with $C^{1,\Ga}$ boundary for some $\Ga>0$.
Note that the NP operator $\Kcal_{\p \GO}[\Gvf](z)$ is defined only for $z \in \p\GO$, in other words, $\Kcal_{\p \GO}$ is an operator on $\p\GO$. If we define the integral for $z$ outside $\p\GO$, it is called the double layer potential, that is,
\beq
\Dcal_{\p \GO}[\Gvf](z) := \int_{\p\GO} \p_{\Gv_w} \GG(w-z) \, \Gvf(w) \, d\Gs(w),\quad z \in \Cbb \setminus \p \GO.
\eeq
The NP operator and the double layer potential enjoy the following jump relation: for $z \in \p\GO$
\beq\label{jump}
\Dcal_{\p \GO}[\Gvf]|_{\pm}(z) := \lim_{t \to +0} \Dcal_{\p \GO}[\Gvf] (z \pm t \nu_z) = \left( \mp \frac{1}{2} + \Kcal_{\p \GO} \right) [\Gvf] (z),
\eeq
where $\nu_z$ denotes the (complexified) outward unit normal vector at $z$ (see, for example, \cite{AK}).

Let $\Ccal_{\p \GO}$ be the Cauchy transform, that is,
\beq\label{Cauchy_transform}
\Ccal_{\p \GO}[\Gvf](z) = \frac{1}{2\Gp i} \int_{\p\GO} \frac{\Gvf(w)}{w-z} \, dw, \quad z \in \Cbb \setminus \p \GO.
\eeq
The following relation holds (see \cite[p. 67]{Shapiro}):
\beq\label{doubleCauchy}
\Dcal_{\p \GO} [\Gvf](z) = \frac{\Ccal_{\p \GO}[\Gvf](z) + \overline{\Ccal_{\p \GO}[\overline{\Gvf}](z)}}{2}.
\eeq
In fact, it can be proved using the relations
$$
\p_{\Gv_z} = \Re{\left[ (\Gv_1 \p_{x} + \Gv_2 \p_{y}) + i(-\Gv_1 \p_{y} + \Gv_2 \p_{x}) \right]} =2\Re{ \left( \Gv_z \p_z \right)}
$$
and $\tau_w d\Gs(w) = dw$, where $\tau_w$ is the unit tangential vector (with the positive orientation) at $w \in \p \GO$. Here $z=x +i y$ and $\p_z= \frac{1}{2}(\p_{x} - i\p_{y})$.

\subsection{The NP operator and Grunsky coefficients}\label{subsec:Grunsky}

From now on we denote the set $\{|z|>R \}$ by $B_R$. Let $\Hcal^{1/2}(\p B_R)$ be the Sobolev space of order $1/2$ and let $\Hcal^{1/2}_0(\p B_R)$ be the subspace of functions of mean value zero, i.e., the collection of $f \in \Hcal^{1/2}(\p B_R)$ of the form $f(z)= \sum_{n \neq 0} a_n (z/R)^n$. For such a function, the homogeneous $\Hcal^{1/2}$ norm is defined by
\beq
\| f \|^2 = \sum_{n \neq 0} |n| |a_n|^2,
\eeq
and it is equivalent to the usual $\Hcal^{1/2}$ norm. For an integer $n \neq 0$, let
\beq
f_n(z)= \frac{1}{\sqrt{|n|}} \left( \frac{z}{R} \right)^n, \quad z \in \p B_R.
\eeq
The functions $f_n$ form an orthonormal basis for $\Hcal^{1/2}_0(\p B_R)$.

Let $\GY$ be the Riemann mapping from $\Cbb \setminus \overline{B_R}$ onto $\Cbb \setminus \overline{\GO}$ of the form \eqnref{exter_Riemann}. Let $\Hcal^{1/2}_0(\p\GO)$ be the collection of functions $g$ such that $g \circ \GY \in \Hcal^{1/2}_0(\p B_R)$, and define the norm on $\Hcal^{1/2}_0(\p\GO)$ by
\beq\label{starnorm}
\| g \|_{*} := \| g \circ \GY \|.
\eeq
Since $\p\GO$ is assumed to be $C^{1,\Ga}$, $\GY$ is $C^{1,\Ga}$ up to $\p\GO$ \cite[Theorem 3.6]{Pommerenke}. Thus, the norm $\| \ \|_*$ is equivalent to the usual $\Hcal^{1/2}$ norm on $\Hcal^{1/2}_0(\p\GO)$.

For each integer $n\neq 0$, define $g_n$ by
\beq\label{g_n1}
g_n(w): = (f_n \circ \GY^{-1})(w), \quad w \in \p\GO.
\eeq
Then, $g_n$, $n \neq 0$, form an orthonormal basis for $\Hcal_0^{1/2}(\p \GO)$ with respect to the inner product induced by the norm $\| \ \|_*$. We note that since $\overline{f_n} = f_{-n}$,
\beq\label{g_n2}
\overline{g_n} = g_{-n}.
\eeq

If $n>0$, $\left( \GY^{-1}(w) \right)^n$ can be uniquely decomposed as
\beq\label{Faber_poly}
\left( \GY^{-1}(w) \right)^n = F_n(w) + \widehat{F}_n(w),
\eeq
where $F_n(w)$ is a polynomial and $\widehat{F}_n(w)$ is an analytic function in $\Cbb \setminus \overline{\GO}$ such that $\widehat{F}_n(w) \to 0$ as $|w| \to \infty$. The function $F_n(w)$ is called the Faber polynomial of degree $n$ generated by $\GY^{-1}$. Let the Laurent series expansion of $\widehat{F}_n (\GY(z))$ be given by
\beq\label{Grunsky_coeff}
\widehat{F}_n(\GY(z)) = -\sum_{k=1}^\infty \frac{c_{n,k}}{z^k},
\eeq
so that
\beq\label{Faber_poly2}
\left( \GY^{-1}(w) \right)^n = F_n(w) -\sum_{k=1}^\infty \frac{c_{n,k}}{(\GY^{-1}(w))^k}
\eeq
The coefficients $c_{n,k}$ are called Grunsky coefficients. See \cite{Curtiss} for the Faber polynomials and the Grunsky coefficients.

The following theorem is obtained in \cite{Jung and Lim}. We include an alternative proof based on the Cauchy integral formula.

\begin{theorem}[\cite{Jung and Lim}]\label{NP_Grunsky}
Let
\beq\label{modified Grunsky}
\Gm_{n,k} = \frac{\sqrt{k}}{2\sqrt{n}}\frac{c_{n,k}}{R^{n+k}}, \quad n,k=1,2, \ldots.
\eeq
It holds that
\begin{align}
\Kcal_{\p \GO}[g_n] &= \sum_{k=1}^{\infty} \Gm_{n,k} \, g_{-k}, \label{Kcalgn1} \\
\Kcal_{\p \GO}[g_{-n}] &= \sum_{k=1}^{\infty} \overline{\Gm_{n,k}} \, g_{k}, \label{Kcalgn2}
\end{align}
for all $n=1,2, \cdots$.
\end{theorem}

\begin{proof}
We assume $R=1$ for simplicity. Since $F_n$ is analytic in $\GO$, $\widehat{F}_n$ is analytic in $\Cbb \setminus \overline{\GO}$ and satisfies $\widehat{F}_n(w) \to 0$ as $|w| \to \infty$, and both of them are continuous up to the boundaries, the Cauchy integral formula and Cauchy's theorem yield that $\Ccal_{\p \GO}[g_n](w)= F_n(w)/|n|^{1/2}$ if $w\in\GO$. Since $\GY^{-1}(w)^{-n} \to 0$ as $|w| \to \infty$, we see that $\Ccal_{\p \GO}[g_{-n}](w) =0$ if $w\in\GO$. It then follows from \eqnref{doubleCauchy} and \eqnref{g_n2} that $\Dcal_{\p \GO}[g_n](w) = F_n(w)/(2|n|^{1/2})$ if $w\in\GO$. We then infer from the jump relation \eqnref{jump} that
$$
\Kcal_{\p \GO}[g_n] = \Dcal_{\p \GO}[g_n]|_- - \frac{1}{2} g_n = - \frac{\widehat{F}_n}{2 \sqrt{|n|}}.
$$
Now, \eqnref{Kcalgn1} follows from \eqref{Grunsky_coeff} and the definition \eqnref{g_n1} of $g_n$. \eqnref{Kcalgn2} follows from \eqnref{g_n2} and \eqnref{Kcalgn1}.
\end{proof}

Theorem \ref{NP_Grunsky} shows that if $g =\sum_{n=1}^{\infty} (a_n g_{n} + b_n g_{-n}) \in \Hcal_{0}^{1/2}(\p \GO)$,
$$
\Kcal_{\p \GO}[g] = \sum_{k=1}^{\infty} \left( \sum_{n=1}^{\infty}\Gm_{n,k} a_n \right) \, g_{-k} + \sum_{k=1}^{\infty} \left( \sum_{n=1}^{\infty}\overline{\Gm_{n,k}} b_n \right)  \, g_{k}.
$$
Thus, if we identify $\Hcal_{0}^{1/2}(\p \GO)$ by $\ell^2 \times \ell^2$ via $g \mapsto (a,b)$ where $a=(a_1,a_2, \ldots)$ and $b=(b_1,b_2, \ldots)$, then the NP operator $\Kcal_{\p \GO}$ is identified with the operator, denoted by $[\Kcal_{\p \GO}]$, on $\ell^2 \times \ell^2$ defined by
\beq\label{NP_matrix_representation}
\left[\Kcal_{\p \GO}\right](a,b) = (\ol{\Mcal} b, \Mcal a),
\eeq
where $\Mcal= \Mcal_{\p \GO}$ is the one-sided infinite matrix defined by
\beq
\Mcal = (\Gm_{n,k})_{n,k \ge 1}^t.
\eeq
Here the superscript $t$ denotes the transpose. In \cite{Jung and Lim}, $\Kcal_{\p \GO}$ is represented by the two-sided infinite matrix.

Here comes a fact of crucial importance: Grunsky theorem says that
\begin{align}
kc_{n,k} = nc_{k,n}, \quad k,n=1,2,\ldots
\end{align}
(\cite{Jabotinsky}), which implies that $\Mcal$ is symmetric, namely,
\beq\label{Msymm}
\Mcal^t=\Mcal.
\eeq
It implies as was shown in \cite{Jung and Lim} that $[\Kcal_{\p \GO}]$ is self-adjoint on $\ell^2 \times \ell^2$, and equivalently, $\Kcal_{\p \GO}$ is self-adjoint on $\Hcal^{1/2}_0(\p\GO)$ equipped with the norm $\| \ \|_{*}$ defined by \eqnref{starnorm}.

We call the constants $\Gm_{n,k}$ defined by \eqref{modified Grunsky} the modified Grunsky coefficients on $\GO$. Sometimes we write $\Gm_{n,k}=\Gm_{n,k}^\GO$ to specify the domain where they are defined.

\subsection{Lemniscates}\label{subsec:lemni}

Let $\GO$ be the lemniscate whose boundary is defined by
\beq\label{analyticset}
\p \GO = \{ z : |P(z)|=R^n \},
\eeq
where $P$ is a polynomial of degree $n$ and $R$ is a number sufficiently large so that all the zeros of $P$ lie inside $\GO$.

The following theorem is proved in \cite{KPS}.
\begin{theorem}[\cite{KPS}]\label{thm:KPS}
If $\GO$ is a lemniscate whose boundary is defined by \eqref{analyticset}, then $\Kcal_{\p \GO}$ has an infinite dimensional kernel.
\end{theorem}

Here we give an alternative proof of this theorem using the representation of $\Kcal_{\p \GO}$ in terms of the Grunsky coefficients.

If $P(z)$ is given by $P(z)= z^n + a_{n-1}z^{n-1} + \cdots+a_0$,  then the Riemann mapping $\GY$ is given by,
\begin{align}\label{Marku}
\GY^{-1}(w) = w \left(1 + \frac{a_{n-1}}{w} + \frac{a_{n-2}}{w^2} + \cdots + \frac{a_{0}}{w^n}\right)^{1/n}
\end{align}
(see \cite{Markusevic}). Since $\left(\GY^{-1}(w)\right)^{nk} = P(w)^k$, $\widehat{F}_{nk}(w)=0$ and hence the Grunsky coefficients $c_{nk,l}$ is $0$ for all $k$ and $l$. Thus $\Kcal_{\p \GO}[g_{nk}]=0$ for all $k \neq 0$ by Theorem \ref{NP_Grunsky}.

\begin{remark}
An $m$th-root transform of a lemniscate is also a lemniscate. In fact, if the lemniscate $\GO$ is given as before. Then its $m$th-root transform $\GO_m$ is given by
$$
\GY_m^{-1}(w) = \GY^{-1}(w^m)^{1/m} = w \left(1 + \frac{a_{n-1}}{w^m} + \frac{a_{n-2}}{w^{2m}} + \cdots + \frac{a_{0}}{w^{nm}}\right)^{\frac{1}{nm}},
$$
namely, it is the lemniscate defined by the polynomial $P(z)= z^{nm} + a_{n-1}z^{(n-1)m} + \cdots+a_0$.
\end{remark}

\section{Proof of Theorem \ref{main}}\label{proof}

Since $\Hcal^{1/2}(\p \GO_m)=\Hcal_0^{1/2}(\p \GO_m) \oplus X$ where $X$ is the one-dimensional subspace of constants, it suffices to prove theorem for $\Hcal_0^{1/2}(\p \GO_m)$.

Let $D$ be a simply connected $m$-fold rotationally symmetric (with respect to the point 0) domain. Let $\GY_D$ be the Riemann mapping from $|z|>R$ for some $R$ onto $\Cbb \setminus \overline{D}$ of the form \eqnref{exter_Riemann}. Let $F_n^{D}(w)$ and $c_{n,k}^{D}$ be the  $n$th Faber polynomial and Grunsky coefficients of $D$ so that the following relation holds
\beq\label{Faber_polynomial_m}
\left( \GY_D^{-1}(w) \right)^n = F_n^{D}(w) - \sum_{k=1}^\infty \frac{c_{n,k}^{D}}{\GY_D^{-1}(w)^k}.
\eeq
Let $\Gm_{n,k}^{D}$ be the modified Grunsky coefficients defined by \eqnref{modified Grunsky}.

\begin{lemma}
Let $D$ be a simply connected $m$-fold rotationally symmetric (with respect to the point 0) domain. Then, for $n,k=1,2,\ldots$,
\beq\label{properties_Grunsky}
\Gm_{n,k}^{D} = 0 \;\; \text{if} \;\; n+k \not\equiv 0 \;(\text{mod} \;\;m).
\eeq
\end{lemma}

\begin{proof}
Let $\Gz_m = e^{2\pi i / m}$. From the rotational symmetry of $D$, $\Gz_m^{-1} \GY_D^{-1}(\Gz_m w) = \GY_D^{-1}(w)$. Thus, we have
$$
\left( \GY_D^{-1}(\Gz_m w) \right)^n = \Gz_m^n\left( \GY_D^{-1}(w) \right)^n
= \Gz_m^n F_n^{D}(w) -\sum_{k=1}^\infty \frac{\Gz_m^n c_{n,k}^{D}}{\left( \GY_D^{-1}(w) \right)^k}.
$$
On the other hands,
$$
\left( \GY_D^{-1}(\Gz_m w) \right)^n = F_n^{D}(\Gz_m w) -\sum_{k=1}^\infty \frac{c_{n,k}^{D}}{\Gz_m^k\left( \GY_D^{-1}(w) \right)^k}.
$$
By equating two identities above, we obtain
$$
c_{n,k}^{D} \left( \Gz_m^{n+k} - 1 \right) = 0,
$$
and hence $c_{n,k}^{D} = 0$ if $n+k \not\equiv 0 \;(\text{mod} \;\;m)$. Thus, \eqnref{properties_Grunsky} follows.
\end{proof}

Let $\GO$ be the simply connected domain whose $m$th-root transform is $D$, namely,
\beq\label{GOmD1}
\GO_m=D.
\eeq
For an existence of such a domain, see \cite{duren} as we mentioned in introduction. Let $\GY_\GO$ be the Riemann mapping from $|z|>R^m$ onto $\Cbb \setminus \overline{\GO}$ such that
\beq\label{GOmD2}
\GY_D(z)=\GY_\GO(z^m)^{1/m}.
\eeq
Let $F_n^{\GO}(w)$ and $c_{n,k}^{\GO}$ be the  $n$th Faber polynomial and Grunsky coefficients of $\GO$ so that the relation \eqnref{Faber_polynomial_m} holds with $D$ replaced with $\GO$. Let $\Gm_{n,k}^{\GO}$ be the modified Grunsky coefficients.

\begin{lemma}
Suppose $D=\GO_m$. Then
\beq\label{cidentity}
\Gm_{{mn},mk}^{D} = \Gm_{{n},k}^{\GO}
\eeq
for all $k,n=1,2,\ldots$.
\end{lemma}

\begin{proof}
One can see from \eqnref{GOmD2} that the following holds:
$$
\GY_D^{-1}(w)^m = \GY_\GO^{-1}(w^m).
$$
We thus have
$$
\left(\GY_D^{-1}(w)\right)^{mn} = \left(\GY_\GO^{-1}(w^m)\right)^{n} = F_n^{\GO}(w^m) -\sum_{k=1}^\infty \frac{c_{n,k}^{\GO}}{\left( \GY_\GO^{-1}(w^m) \right)^k}.
$$
On the other hands, we have from \eqref{properties_Grunsky} that
\begin{align*}
\left(\GY_D^{-1}(w)\right)^{mn} &= F_{mn}^{D}(w) -\sum_{k=1}^\infty \frac{c_{{mn},k}^{D}}{\left( \GY_D^{-1}(w) \right)^k} \\
&= F_{mn}^{D}(w) -\sum_{k=1}^\infty \frac{c_{{mn},mk}^{D}}{\left( \GY_D^{-1}(w) \right)^{mk}} \\
&= F_{mn}^{D}(w) -\sum_{k=1}^\infty \frac{c_{{mn},mk}^{D}}{\left( \GY_\GO^{-1}(w^m) \right)^{k}}.
\end{align*}
Therefore we have $c_{{mn},mk}^{D} = c_{{n},k}^{\GO}$ and \eqnref{cidentity} follows.
\end{proof}

\medskip

\noindent{\sl Proof of Theorem \ref{main}}.
Let $g^D_n$ be the function defined by \eqnref{g_n1}, namely,
\beq\label{g^D_n1}
g^D_n(w): = \frac{1}{\sqrt{|n|}}\left(\frac{\GY_D^{-1}(w)}{R} \right)^n, \quad w \in \p D.
\eeq
For $l=0,1, \cdots, m-1$, we define the subspace $X_{l}$ of $\Hcal_0^{1/2}(\p D)$ by
\begin{align}\label{H_l}
X_{l}:= \overline{\operatorname{span}\{g^D_{n} \;:\; n=jm-l, \ j \in \Zbb \}}
\end{align}
(if $l=0$, $j$ runs in $\Zbb \setminus \{0\}$). Then each $X_{l}$ is invariant under $\Kcal_{\p D}$, \textit{i.e.},
\beq\label{invariant_subspace}
\Kcal_{\p D} \left(X_{l}\right) \subset X_{l}, \quad l=0,1, \ldots, m.
\eeq
In fact, if $g_n^{D} \in X_{l}$ and $n>0$, then $n=jm-l$ for some $j>0$.  Then $\Gm_{n,k} \neq 0$ only when $k=im-(m-l)$ for some $i>0$ by \eqref{properties_Grunsky}. It thus follows from Theorem \ref{NP_Grunsky} that
$$
\Kcal_{\p D}[g_n^{D}] = \sum_{k=1}^{\infty} \Gm_{n,k}^{D} g_{-k}^{D}
= \sum_{i=1}^{\infty} \Gm_{n,im-(m-l)}^{D} g_{-(im-(m-l))}^{D}.
$$
Since $g_{-(im-(m-l))}^{D} \in X_{l}$ by the definition of $X_{l}$, $\Kcal_{\p D}[g_n^{D}] \in X_{l}$. The other case when $n <0$ can be dealt with similarly.

For $j=0,1, \cdots, m_*$, we define the subspace $\Hcal_{j}$ by 
\beq
\Hcal_{j}= 
\begin{cases}
X_j \quad &\mbox{if either $j=0$ or $m$ is even and $j=m/2$}, \\
X_j \oplus X_{m-j} \quad &\mbox{otherwise}.
\end{cases}
\eeq
Then, $\Hcal^{1/2}(\p D)$ is the direct sum of $\Hcal_0, \Hcal_1, \ldots, \Hcal_{m_*}$, namely,
\beq\label{decom3}
\Hcal^{1/2}_0(\p D)= \Hcal_0 \oplus \Hcal_1 \oplus \cdots \oplus \Hcal_{m_*}.
\eeq

One can see that $\Kcal_{\p D}$ is self-adjoint on each $\Hcal_{j}$. This can be seen clearly by the matrix representation of $\Kcal_{\p D}$ on $\Hcal_{j}$ similar to \eqnref{NP_matrix_representation}. In fact, The NP operator $\Kcal_{\p D}$ on $\Hcal_j$, $j=1,\ldots,m_*$, is represented by the operator $[\Kcal_{\p D}|_{\Hcal_j}]$ on $\ell^2 \times \ell^2$ as
\beq\label{j-rep}
\left[\Kcal_{\p D}|_{\Hcal_j} \right](a,b) = (\ol{\Mcal_{j}} b, \Mcal_{j} a),
\eeq
where $\Mcal_j$ is given in \eqnref{Mj} and \eqnref{Mj2} in the next section. Since $\Mcal_{j}$ is symmetric as one can see from its explicit form, $[\Kcal_{\p D}|_{\Hcal_j}]$ is self-adjoint on $\ell^2 \times \ell^2$ and so is $\Kcal_{\p D}$ on $\Hcal_{j}$. In particular, we have the identity \eqnref{decom2}.

Define $U: \Hcal_{0} \to \Hcal_0^{1/2}(\p\GO)$ by $U(g^{D}_{mn})=g^\GO_n$ for all $n \neq 0$. Since both $g^{D}_{mn}$ and $g^\GO_n$ are orthonormal systems, $U$ is a unitary transform.  If $n >0$, then by \eqnref{Kcalgn1}, we have
\begin{align*}
\Kcal_{\p \GO}[U(g^{D}_{mn})] = \Kcal_{\p \GO}[g_n^\GO] = \sum_{k=1}^{\infty} \Gm^\GO_{n,k} g^\GO_{-k} .
\end{align*}
By \eqnref{cidentity}, we have
\begin{align*}
\Kcal_{\p \GO}[U(g^{D}_{mn})] = \sum_{k=1}^{\infty} \Gm_{{mn},mk}^{D} g^\GO_{-k} =  \sum_{k=1}^{\infty} \Gm_{{mn},mk}^{D} U(g^{D}_{-mk}).
\end{align*}
It then follows from \eqnref{properties_Grunsky} that
\begin{align*}
\Kcal_{\p \GO}[U(g^{D}_{mn})] = U \Kcal_{\p D}[g^{D}_{mn}].
\end{align*}
The case for $n \le 0$ can be dealt with in the same way. This completes the proof.

\section{Matrix representations}

Let $D$ be a simply connected $m$-fold rotationally symmetric domain and let $\GO$ be the domain such that $D=\GO_m$. The matrix representations of $\Kcal_{\p D}$ would yield a clear picture on the spectral properties.

As in \eqnref{NP_matrix_representation}, $\Kcal_{\p D}$ is represented by the operator $[\Kcal_{\p D}]$ on $\ell^2 \times \ell^2$ defined by
$$
\left[\Kcal_{\p D}\right](a,b) = (\ol{\Mcal_{\p D}} b, \Mcal_{\p D} a),
$$
where $\Mcal_{\p D}= (\Gm_{n,k}^D)_{n,k \ge 1}^t$. According to \eqnref{properties_Grunsky},
$\Mcal_{\p D}$ takes the form
$$
\left[\begin{array}{ccc|c|ccc|c|c}
	      0        &         &  \Gm_{1,m-1}^D  &              &        0         &         &  \Gm_{1,2m-1}^D  &               &        \\
	               & \iddots &                 &      0       &                  & \iddots &                  &       0       & \cdots \\
	\Gm_{m-1,1}^D  &         &        0        &              & \Gm_{m-1,m+1}^D  &         &        0         &               &        \\ \hline
	               &    0    &                 & \Gm_{m,m}^D  &                  &    0    &                  & \Gm_{m,2m}^D  & \cdots \\ \hline
	      0        &         & \Gm_{m+1,m-1}^D &              &        0         &         & \Gm_{m+1,2m-1}^D &               &        \\
	               & \iddots &                 &      0       &                  & \iddots &                  &       0       & \cdots \\
	\Gm_{2m-1,1}^D &         &        0        &              & \Gm_{2m-1,m+1}^D &         &        0         &               &        \\ \hline
	               &    0    &                 & \Gm_{2m,m}^D &                  &    0    &                  & \Gm_{2m,2m}^D & \cdots \\ \hline
	               & \vdots  &                 &    \vdots    &                  & \vdots  &                  &    \vdots     & \ddots
\end{array} \right].
$$
Thus $\Mcal_j$ represneting $\Kcal_{\p D}|_{\Hcal_j}$ as in \eqnref{j-rep} is given as follow:

\begin{itemize}
\item[(i)] If $j=0$ or $j=m/2$ ($m$ is even), then
\begin{align}
\Mcal_{j} = \left[\begin{array}{cccc}
	 \Gm_{m-j,m-j}^{D}   &  \Gm_{m-j,2m- j}^{D}   &  \Gm_{m-j,3m-j}^{D}   & \cdots \\
	\Gm_{2m-j,m-j}^{D}  & \Gm_{2m-j,2m-j}^{D}  & \Gm_{2m-j,3m-j}^{D}  & \cdots \\
	\Gm_{3m-j,m-j}^{D} & \Gm_{3m-j,2m-j}^{D} & \Gm_{3m-j,3m-j}^{D} & \cdots \\
	      \vdots       &       \vdots        &       \vdots        & \ddots
\end{array}\right]. \label{Mj}
\end{align}

\item[(ii)] Otherwise,
\begin{align}
\Mcal_{j} = \left[\begin{array}{cc|cc|c}
	     0       &  \Gm^D_{j,m-j}  &       0        &  \Gm^D_{j,2m-j}  & \cdots \\
	\Gm^D_{m-j,j}  &       0       & \Gm^D_{m-j,m+j}  &       0        & \cdots \\ \hline
	     0       & \Gm^D_{m+j,m-j} &       0        & \Gm^D_{m+j,2m-j} & \cdots \\
	\Gm^D_{2m-j,j} &       0       & \Gm^D_{2m-j,m+j} &       0        & \cdots \\ \hline
	   \vdots    &    \vdots     &     \vdots     &     \vdots     & \ddots
\end{array} \right]. \label{Mj2}
\end{align}

\end{itemize}
Note that
\beq
\Mcal_{0} = \Mcal_{\p \GO},
\eeq
which is due to \eqnref{cidentity}.

We note that $\Mcal_{j}^t = \Mcal_{j}$. So, if $\Mcal_{j}$ is real, then it admits the diagonalization, and its eigenvalues completely determines those of $\Kcal_{\p D}|_{\Hcal_j}$. In fact, we have the following proposition.

\begin{prop}\label{prop2}
Suppose that $\Mcal_{j}$ is real. If $\Gl$ is an eigenvalue of $\Kcal_{\p D}|_{\Hcal_j}$, then either $\Gl$ or $-\Gl$ is an eigenvalue of $\Mcal_{j}$. Conversely, if $\Gl$ is an eigenvalue of $\Mcal_{j}$, then both $\Gl$ and $-\Gl$ are eigenvalues of $\Kcal_{\p D}|_{\Hcal_j}$.
\end{prop}

\begin{proof}
Suppose that $\Gl$ is an eigenvalue of $\Kcal_{\p D}|_{\Hcal_j}$. According to \eqnref{j-rep}, there is $a,b \in \ell^2$ (not both zero) such that
$$
\Mcal_{j} b=\Gl a, \quad \Mcal_{j} a=\Gl b.
$$
Thus we have
$$
\Mcal_{j} (a+b)=\Gl (a+b), \quad \Mcal_{j} (a-b)=-\Gl (a-b).
$$
Since either $a+b$ or $a-b$ is nonzero, we infer that $\Gl$ or $-\Gl$ is an eigenvalue of $\Mcal_{j}$.

If $\Gl$ is an eigenvalue of $\Mcal_{j}$, then there is nonzero $a \in \ell^2$ such that
$$
\Mcal_{j}a = \Gl a.
$$
Therefore, we have
\begin{align*}
&\left[ \Kcal_{\p D}|_{\Hcal_j}\right](a,a) = (\Mcal_{j}a, \Mcal_{j}a) = \Gl(a,a), \\
&\left[ \Kcal_{\p D}|_{\Hcal_j}\right](a,-a) = (-\Mcal_{j}a, \Mcal_{j}a) = -\Gl(a,-a).
\end{align*}
We conclude that both $\Gl$ and $-\Gl$ are eigenvalues of $\Kcal_{\p D}|_{\Hcal_j}$.
\end{proof}

\section{Examples}

\subsection{$m$-star-shaped domains}
For a positive integer $m$, let $S_m$ be the regular $m$-star, namely,
\begin{align}
S_m = \{x\Gz_m^k : 0\leq x \leq 4^{1/m}, \ k=0,1, \ldots, m-1 \},
\end{align}
where $\Gz_m = e^{2\pi i/m}$. Let
$$
\GY_m(z) = z\left(1+\frac{1}{z^m}\right)^{\frac{2}{m}}, \quad m=1,2,\ldots.
$$
It is known (see \cite{Bartalome and He, Henrici}) that $\GY_m$
maps $|z|>1$ conformally onto $\Cbb \setminus S_m$. The mapping $\GY_m$ is the $m$th-root transform of $\GY_1$, that is, $\GY_m(z)=\GY_1(z^m)^{1/m}$.

\begin{figure}[h]
	\centering
	\begin{subfigure}[h]{0.25\textwidth}
		\includegraphics[scale=0.3]{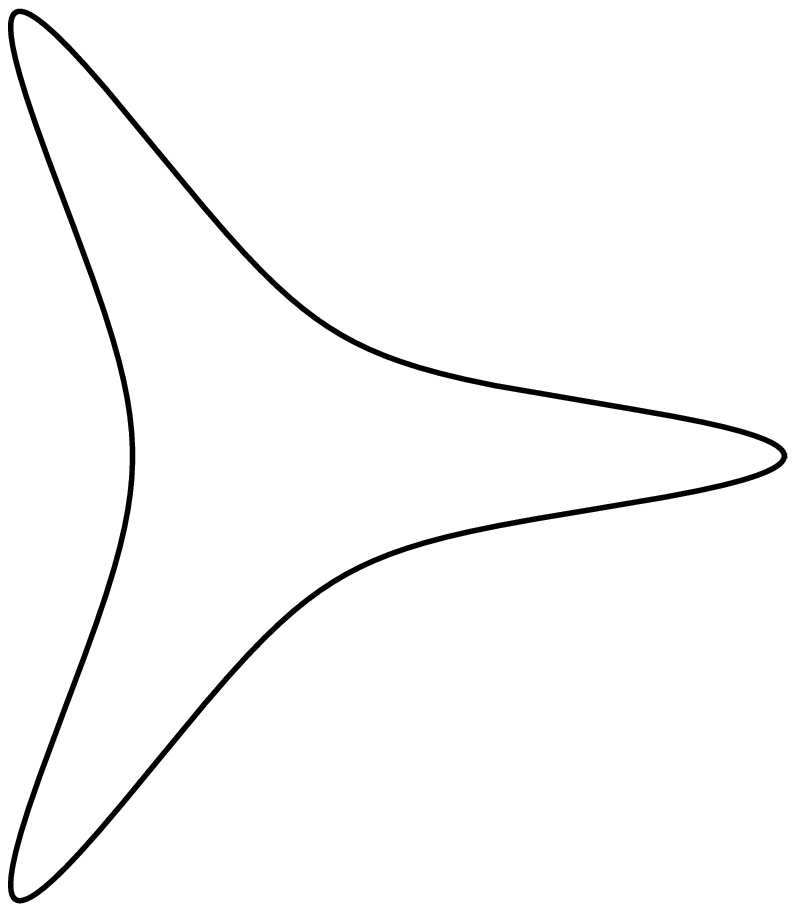}
		\centering
	\end{subfigure}
	\begin{subfigure}[h]{0.25\textwidth}
		\includegraphics[scale=0.3]{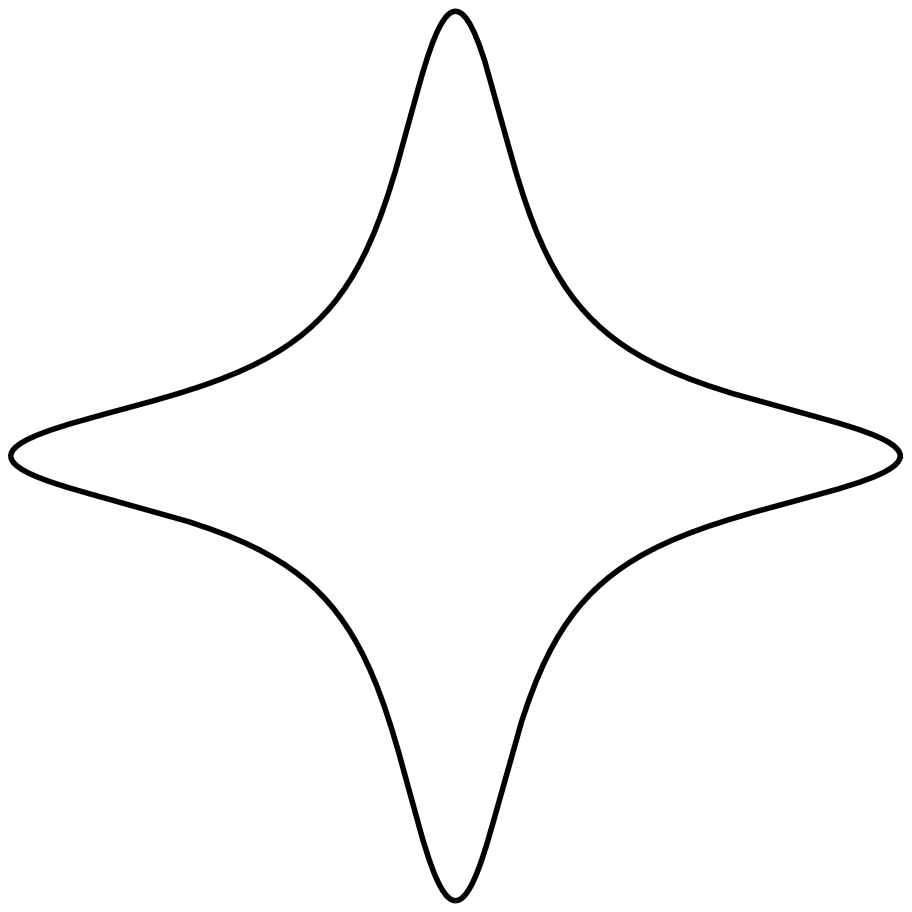}
		\centering
	\end{subfigure}
	\begin{subfigure}[h]{0.25\textwidth}
		\includegraphics[scale=0.3]{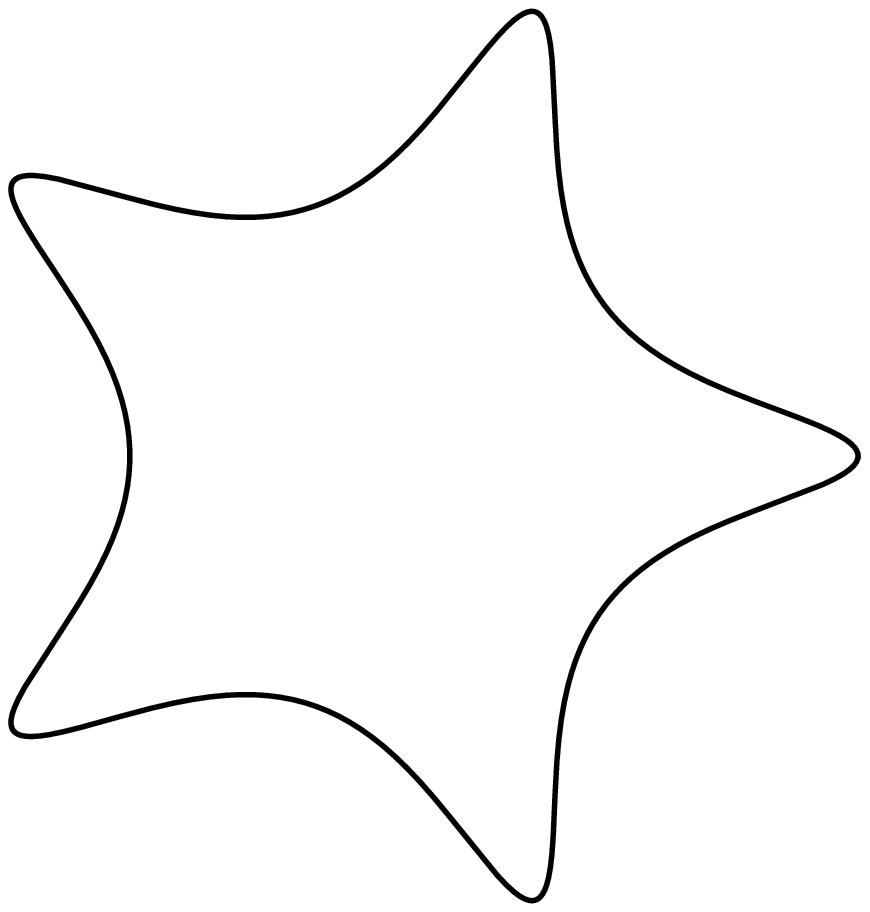}
		\centering	
    \end{subfigure}
	\caption{$m$-star-shaped domains $\p \GO_m$, $m=3,4,5$ (from left to right), which are images of $\{ |z|=R^{1/m}$ under $\Psi_m$. Here, $R=1.1$. All of them have $\pm R^{-2n}/2$ ($n =1,2,\ldots$) as their NP eigenvalues.}\label{fig1}
\end{figure}

Fix $R >1$ and let $\GO_{m}$ be bounded domains such that $\p \GO_{m}= \{ \GY_m(z) : \left|z\right| = R^{1/m}\}$. See Figure \ref{fig1} for shapes of $\p\GO_m$. It is known that the eigenvalues of the NP operator on the ellipse $x^2/a^2+y^2/b^2=1$ ($a \ge b$) are $\pm (a-b)^n/2(a+b)^n$, $n=1,2,\ldots,$ and 1/2 (see, for example, \cite[Proposition 8]{KPS}). Thus one can see easily that $\Gs(\Kcal_{\p \GO_1}) = \{ \pm R^{-2n}/2 \; : \; n=1,2, \ldots \} \cup \{1/2\}$.
Since $\GO_m$ is the $m$th-root transform of $\GO_1$, we infer from the Theorem \ref{main} that $\Gs(\Kcal_{\p \GO_{m}})$ contains this set.  Note that $\p \GO_2$ is an ellipse and $\Gs(\Kcal_{\p \GO_{2}}) = \{ \pm R^{-n}/2 \; : \; n=1,2, \ldots \} \cup \{1/2\}$.
Thus $\Gs(\Kcal_{\p \GO_{2}}) = \Gs(\Kcal_{\p \GO_{1}}) \cup \Gs(\Kcal_{\p \GO_{2}}, \Hcal_1)$ and $\Gs(\Kcal_{\p \GO_{2}}, \Hcal_1) = \{ \pm R^{-2n+1}/2 \; : \; n=0,1, \ldots \}$.

\subsection{Cassini oval}

A Cassini oval $D$ is a lemniscate defined by \eqref{analyticset} with the polynomial $P(z)= z^2-1$. See Figure \ref{fig2}.  The Riemann mappings $\GY_D$ for $D$ is given by
\begin{align}
\GY_D^{-1}(w)&=w\left(1-\frac{1}{w^2}\right)^{\frac{1}{2}}.
\end{align}
Note that $\GY_D$ is the 2nd-root transform of $\GY(z)= z+1$ and hence $D$ is the 2nd-root transform of $\GO=\{ |w-1|=R^{2} \}$. So $\Gs(\Kcal_{\p D})$ contains of $0$ as an eigenvalue of infinite multiplicity since $\Gs(\Kcal_{\p \GO})$ does. This fact is already known by Theorem \ref{thm:KPS}. Here, we look into $\Gs(\Kcal_{\p D}, \Hcal_1)$, spectrum of $\Kcal_{\p D}$ on the invariant subspace $\Hcal_1$.

\begin{figure}[h]
	\centering
	\includegraphics[scale=0.4]{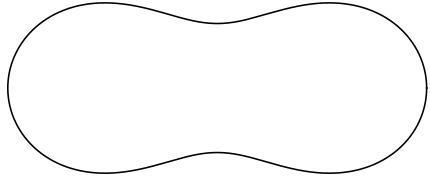}
	\caption{Cassini oval when $R=1.1$.}\label{fig2}
\end{figure}

If $n>0$, then
\begin{align*}
\left(\GY_D^{-1}(w)\right)^{2n+1} = w^{2n+1} \left(1-\frac{1}{w^2}\right)^{\frac{2n+1}{2}} 
= w^{2n+1} \sum_{j=0}^{\infty} (-1)^j \binom{\frac{2n+1}{2}}{j} \frac{1}{w^{2j}}. 
\end{align*}
Therefore, the Faber polynomial $F_{2n+1}(w)$ is given by
\begin{align}
F_{2n+1}(w) = \sum_{j=0}^{n} (-1)^{n-j} \binom{\frac{2n+1}{2}}{n-j}  w^{2j+1}. 
\end{align}
It then follows that 
\begin{align*}
F_{2n+1}(\GY_D(z)) = &\sum_{j=0}^{n} (-1)^{n-j} \binom{\frac{2n+1}{2}}{n-j} \left[z^{2j+1}\left(1+\frac{1}{z^2}\right)^{\frac{2j+1}{2}}\right] \\
= &\sum_{j=0}^{n} (-1)^{n-j} \binom{\frac{2n+1}{2}}{n-j} \left[z^{2j+1} \sum_{k=0}^{\infty} \binom{\frac{2j+1}{2}}{k} \frac{1}{z^{2k}}\right] \\
= & \sum_{k=0}^{n} \left(\sum_{j=k}^{n} (-1)^{n-j} \binom{\frac{2n+1}{2}}{n-j}  \binom{\frac{2j+1}{2}}{j-k}\right) z^{2k+1}\\
&+ \sum_{k=0}^{\infty}  \left(\sum_{j=0}^{n} (-1)^{n-j} \binom{\frac{2n+1}{2}}{n-j} \binom{\frac{2j+1}{2}}{k+j+1}\right) \frac{1}{z^{2k+1}}.
\end{align*}
So, the Grunsky coefficients are given by
\begin{align}
	c_{2n+1,2k+1} = \sum_{j=0}^{n} (-1)^{n-j} \binom{\frac{2n+1}{2}}{n-j}  \binom{\frac{2j+1}{2}}{k+j+1}.
\end{align}
In particular, they are real. Thus, eigenvalues of $\Mcal_1$ (and numbers of the opposite sign) are eigenvalues of $\Kcal_{\p D}$ on $\Hcal_1$ by Proposition \ref{prop2}.

By the definition \eqnref{modified Grunsky} of the modified Grunsky coefficients, we see that
the matrix $\Mcal_{1}$ is of the form
\begin{align}
\Mcal_{1} = \ds \left[\begin{array}{cccc}
	     \frac{1}{2}\frac{1}{R^2}      &   -\frac{\sqrt{3}}{8} \frac{1}{R^4}   &  \frac{\sqrt{5}}{16} \frac{1}{R^6}   & \cdots
\medskip \\
	-\frac{\sqrt{3}}{8}\frac{1}{R^4} &       \frac{1}{8}\frac{1}{R^6}        & -\frac{3\sqrt{15}}{128}\frac{1}{R^7} & \cdots \medskip \\
	\frac{\sqrt{5}}{16}\frac{1}{R^6} & -\frac{3\sqrt{15}}{128} \frac{1}{R^8} &     \frac{9}{128}\frac{1}{R^{10}}     & \cdots \medskip \\
	             \vdots              &                \vdots                 &                \vdots                &     \ddots
\end{array}\right].
\end{align}
It is not clear whether the eigenvalues of $\Mcal_1$ can be computed explicitly. But, eigenvalues can be computed numerically. Moreover, since $R>1$, finite submatrices yield good approximations of eigenvalues as Table \ref{table} shows.

Let $\left[ \Mcal_{1} \right]_{n}$ be the submatrix of $\Mcal_{1}$ of size $n\times n$ obtained by taking the first $n$ rows and columns. Let $\Gl_j$ be eigenvalues enumerated according to the rule $|\Gl_1| \geq |\Gl_2| \geq \cdots $. Table \ref{table} exhibits the first 10 eigenvalues when $n=10, 25, 50, 100$. There eigenvalues are computed by using Python built-in function \textit{eigvalsh} for numerical computation on eigenvalues of Hermitian matrix.

\begin{table}[h]
\center \begin{tabular}{|c||c|c|c|c|}
	\hline
	           & $\left[ \Mcal_{1} \right]_{10}$ & $\left[\Mcal_{1} \right]_{25}$ & $\left[ \Mcal_{1} \right]_{50}$ & $\left[ \Mcal_{1} \right]_{100}$ \\ \hline\hline
	 $\Gl_1$   &            0.249194             &            0.249279            &            0.249280             &             0.249280             \\ \hline
	 $\Gl_2$   &            0.0188675            &            0.019039            &            0.0190397            &            0.0190397             \\ \hline
	 $\Gl_3$   &           0.00126322            &           0.00135824           &           0.00135840            &            0.00135840            \\ \hline
	 $\Gl_4$   &          0.0000716940           &          0.0000967768          &          0.0000968816           &           0.0000968816           \\ \hline
	 $\Gl_5$   &    3.10066 $\times 10^{-6}$     &   6.86233  $\times 10^{-6}$    &   6.90960    $\times 10^{-6}$   &    6.90960   $\times 10^{-6}$    \\ \hline
	 $\Gl_6$   &   9.74582   $\times 10^{-8}$    &   4.77676   $\times 10^{-7}$   &   4.92791    $\times 10^{-7}$   &    4.92793   $\times 10^{-7}$    \\ \hline
	 $\Gl_7$   &    2.16813  $\times 10^{-9}$    &   3.16948   $\times 10^{-8}$   &   3.51450    $\times 10^{-8}$   &    3.51461   $\times 10^{-8}$    \\ \hline
	 $\Gl_8$   &   3.24992   $\times 10^{-11}$   &   1.93082  $\times 10^{-9}$    &   2.50618    $\times 10^{-9}$   &    2.50662   $\times 10^{-9}$    \\ \hline
	 $\Gl_9$   &   2.94949   $\times 10^{-13}$   &   1.05077  $\times 10^{-10}$   &  1.78615    $\times 10^{-10}$   &   1.78772   $\times 10^{-10}$    \\ \hline
	$\Gl_{10}$ &  1.22538     $\times 10^{-15}$  &  5.05405   $\times 10^{-12}$   &  1.27027    $\times 10^{-11}$   &   1.27501   $\times 10^{-11}$    \\ \hline
\end{tabular}
\caption{Eigenvalues of $\left[ \Mcal_{1} \right]_{n}$, $n=10, 25, 50, 100$, when $R=1.1$.}\label{table}
\end{table}


\end{document}